\documentclass[12pt]{article}
\usepackage[a4paper, total={17cm,25cm}]{geometry}

\usepackage{amsmath,amssymb,amsfonts,amsthm}
\usepackage{graphicx}
\usepackage[alphabetic,y2k,initials]{amsrefs}

\graphicspath{ {./ImagesForFigures/} }

\usepackage{color, graphics}
\usepackage{float}
\usepackage{tikz}
\usetikzlibrary{arrows}
\usepackage{pdfsync}
\usepackage{hyperref}
\usepackage{multirow}

\usepackage[shortlabels]{enumitem}

\usepackage{calrsfs}
\DeclareMathAlphabet{\pazocal}{OMS}{zplm}{m}{n}

\newcommand{\C}{\pazocal{C}}

\newcommand{\E}{\pazocal{E}}
\newcommand{\F}{\pazocal{F}}
\newcommand{\K}{\pazocal{K}}

\newcommand{\Curve}{\mathcal{C}}

\renewcommand{\L}{\pazocal{L}}

\renewcommand{\O}{\pazocal{O}}

\newcommand{\refeq}{\eqref}

\theoremstyle{plain}
\newtheorem{theorem}{Theorem}[section]
\newtheorem{proposition}[theorem]{Proposition}

\theoremstyle{definition}
\newtheorem{example}[theorem]{Example}

\theoremstyle{remark}
\newtheorem{remark}[theorem]{Remark}

\newcommand{\fourbyfour}[9]{\left( \begin{array}{cccc}
		#1 & #2 & \cdots & #3 \\
		#4 & #5 & \cdots & #6\\
		\vdots & \vdots & \cdots  & \vdots \\
		#7 & #8 & \cdots & #9 \end{array} \right)}

\numberwithin{equation}{section}

 \title{Periodic Trajectories and Topology of the Integrable Boltzmann System}

\usepackage{authblk}
\author[1]{Sean Gasiorek}
\author[2,3]{Milena Radnovi\'c}
\affil[1]{\textsc{California Polytechnic State University, San Luis Obispo}}
\affil[2]{\textsc{The University of Sydney, School of Mathematics and Statistics}}
\affil[3]{\textsc{Mathematical Institute SANU, Belgrade}}
\affil[ ]{\texttt{gasiorek@calpoly.edu, milena.radnovic@sydney.edu.au}}

\date{}

\begin{document}

\maketitle

\begin{abstract}
We consider the Boltzmann system corresponding to the motion of a billiard with a linear boundary under the influence of a gravitational field. We derive analytic conditions of Cayley's type for periodicity of its trajectories and provide geometric descriptions of caustics. 
The topology of the phase space is discussed using Fomenko graphs.
	
\smallskip

\emph{Keywords:} Integrable Boltzmann system, Kepler problem, billiards, periodic trajectories, Poncelet theorem, Fomenko graphs

\smallskip

\textbf{MSC2020:} 37J35, 37J39, 37J46, 37C79, 37C83, 70G40 
\end{abstract}


\section{Introduction}

The Boltzmann system \cite{Boltzmann} consists of a massive particle moving in a gravitational field with a linear boundary (i.e. the wall) between the particle and the centre of gravity.
The reflections off the boundary are absolutely elastic, meaning that the kinetic energy remains unchanged by them, and they obey the billiard reflection law, i.e.~the angles of incidence and reflection are congruent to each other.

It was recently shown in \cite{GJ20} that this system is integrable, since it has, in addition to the energy, a second integral of motion. That additional integral can be geometrically seen in the fact that, for each trajectory, there is a fixed circle such that all arcs of the trajectory are Kepler conics with one focus at the centre of gravity and the second focus on that fixed circle.
In \cite{F21}, this system is analysed further, and it is proved that the Boltzmann map is equivalent to a shift on a given elliptic curve, which then implies a Poncelet-style closure theorem in this setting:
if a given trajectory of the Boltzmann system is $n$-periodic, then each trajectory consisting of arcs with foci on the same circle as the initial $n$-periodic trajectory is also $n$-periodic.

This Poncelet-style closure result is the initial point for this paper, from which we find the analytic conditions for periodicity of the trajectories of the Boltzmann system. Those conditions and examples are presented in Section \ref{sec:periodic}.
In Section \ref{sec:geom}, we discuss the geometry of Boltzmann trajectories. 
We show the existence of caustics and the focal reflection property in Theorems \ref{th:caustics} and \ref{th:focal}.
In Section \ref{sec:phase}, the phase space of the Boltzmann system is analysed.
We describe in detail the dynamics on the singular level sets in Theorem \ref{th:singular-level-sets} and, using the Fomenko graphs and invariants, we provide topological description of the isoenergy manifolds in Theorem \ref{th:fomenko}.

Before we proceed to those discussions, we will briefly recall, following \cite{F21}, the construction and recent results of the Boltzmann system.

\subsection{Integrability of the Boltzmann billiard}

Here we review of notions and results from \cite{F21}, which we will immediately use in Section \ref{sec:periodic} to derive analytical conditions for periodicity of the Boltzmann system.

In the classical $2$-body Kepler problem with an inverse square central force law, solutions to the reduced problem in some inertial reference frame are conics with one focus at the origin. The position $\textbf{r}$ and linear momentum $\textbf{p}$ of the reduced body relate to the angular momentum $\textbf{L}$ by 
$$
\textbf{L} = \textbf{r}\times\textbf{p},
$$
which is an integral of motion and defines the plane of motion. 
The total energy $E$ and  Laplace--Runge--Lenz vector $\textbf{A} = \textbf{p} \times \textbf{L}-\dfrac{\mathbf{r}}{|\mathbf{r}|}$ are also integrals of motion. The conic is an ellipse if $E <0$, a parabola if $E=0$, and one branch of a hyperbola if $E>0$.  

Following the notation from \cite{F21}, we consider the Boltzmann system in the plane with coordinates $x_1, x_2$ with the centre of gravity at the origin and wall at $x_2 =1$.
Vector $\mathbf{A}$ will then also be in that plane and we denote its components by $A_1$, $A_2$, while $\mathbf{L}$ is orthogonal to the plane and we denote its third component by $L$.

It was shown in \cite{GJ20} that 
$$
D := L^2 - 2A_2
$$
is an integral of motion for the Boltzmann system. 
In the coordinates $(x_1,x_2)$, the Kepler conic $\K$ is given by 
\begin{equation}\label{eq:KeplerConic}
	\K: \quad x_1^2 + x_2^2 = (D+2A_2 - A_1 x_1 - A_2 x_2)^2.
\end{equation}
One focus of $\K$ is at the origin and the other one is $F_2:=(A_1/E, A_2/E)$ that lies on a circle $\C$ of radius $R/|E|$ with 
$R^2 := 1+2DE + 4E^2$ centred at the point $\F (0,2)$.

\begin{remark}\label{rem:OF}
	The point $\F(0,2)$, the centre of the circle $\C$, is symmetric to the origin $\O(0,0)$ with respect to the wall.
	We will show in Section \ref{sec:geom} that this point will play other significant roles in the dynamics of the Boltzmann system, see Theorems \ref{th:caustics} and \ref{th:focal}.
\end{remark}

The pair $(E,D)$ determines the level set $X(D,E)$ of the system.
According to \cite[Theorem 2]{F21}, the level set $X(D,E)$ has a compactification that is a smooth projective curve of genus $1$ whenever 
the following conditions hold:
\begin{equation}\label{eq:singular-conditions}
	D^2\neq4,\quad 1+2ED+4E^2\neq0,\quad D+2E\neq0.	
\end{equation}
Such level sets correspond to the non-degenerate Liouville tori.
If some of the inequalities \refeq{eq:singular-conditions} are not satisfied, the level set $X(D,E)$ will be singular.
We note that in Section \ref{sec:periodic}, where our goal is to find analytic conditions for periodicity, only non-singular level sets are of interest, so we will assume there that the conditions \refeq{eq:singular-conditions} hold. 
On the other hand, the singular level sets will be analysed in more detail in Section \ref{sec:phase}.

In the non-degenerate case, the elliptic curve corresponding to $X(D,E)$ is:
\begin{equation}\label{eq:curve}
	y^2=(1-s^2)(1-k^2s^2),
	\quad\text{with}\quad
	k^2=\frac{D+4E-2R}{D+4E+2R}.
\end{equation}
That elliptic curve can also be seen in the affine space with coordinates $(x,A_1,A_2)$ as:
$$
A_1^2+A_2^2-4EA_2=1+2DE,
\quad
x^2+1=(A_2+D-A_1x)^2,
$$
where $(x,1)$ is a reflection point on the wall and $\mathbf{A}=(A_1,A_2)$ is the corresponding Laplace-Runge-Lenz vector.
The Boltzmann map is given in those coordinates as a composition of two involutions $i$ and $j$:
\begin{equation}\label{eq:involutions}
	i(x,A_1,A_2)=(x',A_1,A_2),
	\quad
	j(x,A_1,A_2)=(x,A_1',A_2'),
\end{equation}
with
$$
x'=\frac{-2(A_2+D)A_1}{1-A_1^2}-x,\quad
A_1'=\frac{(x^2-1)A_1-2xA_2+4xE}{x^2+1},\quad
A_2'=\frac{-2xA_1-(x^2-1)A_2+4x^2E}{x^2+1}.
$$
The involution $i$ maps one intersection point of the Kepler ellipse to the other one.
On the other hand, the involution $j$ preserves the intersection point, but maps the vector $\mathbf{A}$ before the reflection to the vector $\mathbf{A}'$ after the reflection, i.e.~switches to the next Kepler ellipse arc of the Boltzmann trajectory.

Both involutions $i$, $j$ have fixed points \cite{F21}, thus they represent point reflections on the Jacobian of the elliptic curve, and 
the Boltzmann map, as their composition $j\circ i$, is equivalent to  following shift of the Jacobian of that curve:
$$
u\mapsto u+\int_{-1}^{s_0}\frac{ds}y,
\quad\text{where}\quad
s_0=\frac{D+2E+R}{D+2E-R},
$$
which is the shift by the divisor $Q_+-P_+$, where points $P_+$, $Q_+$ are given by:
\begin{equation}\label{eq:P+Q+}
P_+(-1,0),
\quad
Q_+\left(
\dfrac{D+2E+R}{D+2E-R},
\dfrac{-4i(D+2E)R}{(D+2E-R)^2\sqrt{(D + 2E)(D + 4E + 2R)}}
\right).
\end{equation}
Thus the $n$-periodicity of any trajectory on $X(D,E)$ is equivalent to:
$$
n(Q_+-P_+)\sim0.
$$
Since that condition does not depend on the initial point of motion, but only on the constants $D$, $E$, we will have the Poncelet property: $n$-periodicity of one trajectory of the Boltzmann system implies that all trajectories with the same constants $D$, $E$ of motion will also be $n$-periodic \cite{F21}.

\section{Periodic trajectories of the Boltzmann system}\label{sec:periodic}

In this section, we will first make a very brief review of results connected with closure theorems of Poncelet type and the corresponding Cayley-type conditions, in particular in the context of billiards.
After that, we provide in Theorem \ref{th:cayley} the analytic conditions for periodicity of the Boltzmann billiard, and then illustrate it by several examples.

We note that closure theorems of Poncelet type and the corresponding analytic conditions originate in classical works of XIXth century mathematicians. 
Namely, the classical Poncelet porism states that the existence of a closed polygon inscribed in one conic and circumscribed about another one implies the existence of infinitely many such polygons;
moreover, each point of the circumscribed conic is a vertex of one of them and all of those polygons have the same number of sides  \cite{Poncelet}.
While Poncelet's approach was synthetic geometric, Jacobi gave alternative proof using addition formulas for elliptic functions, see \cites{Jacobi, Schoenberg}.
Explicit analytic conditions for closure were derived by Cayley \cite{Cayley}.
Modern algebro-geometric approach to Poncelet theorem and those analytic conditions can be found in \cites{GH1977,GH1978}.

The interest in the Poncelet theorem, its generalizations and its applications has been renewed in the last few decades, and there is a large body of works regarding that.
A recent detailed account on the history of Poncelet theorem can be found in \cites{DC2016a,DC2016b}.
More about analytic conditions for the Poncelet theorem and the generalizations, with a review of modern development in the theory of billiards can be found in \cites{DR2011,DR2014}.
Among various generalisations of the Poncelet theorem, those in distinct geometries or those where a potential field is present were developed, see for example \cites{Veselov1990,DJR2003,GKT2007,KT2009,GR2021} and \cites{AF2006,Fed2001}.
For even more results, see references therein.

The conditions for $n$-periodicity of the Boltzmann billiard are given in the following theorem.

\begin{theorem}\label{th:cayley}
The trajectories of the Boltzmann system with integrals $D$ and $E$ satisfying \refeq{eq:singular-conditions} are $n$-periodic if and only if
	\begin{equation*}
	\det\fourbyfour{B_{3}}{B_{4}}{B_{m+1}}{B_{4}}{B_{5}}{B_{m+2}}{B_{m+1}}{B_{m+2}}{B_{2m-1}}=0
\quad\text{with}\quad
n=2m \geq 4,
	\end{equation*}
or
	\begin{equation*}
	\det \fourbyfour{B_{2}}{B_{3}}{B_{m+1}}{B_{3}}{B_{4}}{B_{m+2}}{B_{m+1}}{B_{m+2}}{B_{2m}}=0
\quad\text{with}\quad
n=2m+1\geq3,
	\end{equation*}
where $B_0$, $B_1$, $B_2$, $B_3$, \dots are the coefficients in the Taylor expansion of:
$$\sqrt{\left(2(D+2E)\xi-R\right)\left(4R(D+2E)^2\xi^2 + 2(D+2E)(D^2+2DE-2)\xi +R\right)}$$
around $\xi=0$.
\end{theorem}

\begin{proof}
The algebro-geometric condition for $n$-periodicity is $n(P_+-Q_+)\sim0$, with points $P_+$, $Q_+$ given by \eqref{eq:P+Q+} on the curve \eqref{eq:curve} \cite{F21}.

We note that $P_+$ is a branching point of that curve.
To simplify the calculations, we make the coordinate transformation $(s,y)\to(\xi,\eta)$ such that $\xi(P_+)=\infty$, $\xi(Q_+)=0$:
$$
\xi=\frac1{s+1}-\frac{D+2 E-R}{2 D+4 E},
\quad
\eta=\frac{y}{(s+1)^2}\cdot(D+2E)\sqrt{(D + 2E)(D + 4E + 2R)}.
$$
In these new coordinates, the elliptic curve \eqref{eq:curve} becomes
$$
\Curve:
\quad
\eta^2 =\left(2(D+2E)\xi-R\right)\left(4R(D+2E)^2\xi^2 + 2(D+2E)(D^2+2DE-2)\xi +R\right).
$$
We now derive the Cayley-type conditions similarly as in \cite{GH1978}.

The divisor condition $n(Q_+ - P_+)\sim 0$ is equivalent to the existence of a meromorphic function on $\Curve$ with a unique pole at $P_+$ and unique zero at $Q_+$ such that the pole and zero are both of multiplicity $n$. 
We denote by $\L(nP_+)$ the linear space of all meromorphic functions on $\Curve$ which have a pole of order at most $n$ at $P_+$ and are holomorphic elsewhere.
A basis of $\L(nP_+)$ is:
$$
\begin{aligned}
&\left\{1, \xi, \xi^2, \ldots, \xi^m, \eta, \eta \xi, \ldots, \eta \xi^{m-2} \right\},
\quad\text{for}\quad n=2m;
\\
&\left\{1, \xi, \xi^2, \ldots, \xi^m, \eta, \eta \xi, \ldots, \eta \xi^{m-1} \right\},
\quad\text{for}\quad n=2m+1.
\end{aligned}
$$
It can be derived, in the same way as it is done in \cite{GH1978}, that there is a nontrivial linear  combination of these functions with a zero of order $n$ at $\xi=0$ if and only if the stated determinant conditions hold. 
\end{proof}

Now, we use the analytic conditions from Theorem \ref{th:cayley} to construct examples of periodic trajectories.

\begin{example}[Period 3] \label{ex:Period3}
The Cayley-type condition for a $3$-periodic trajectory is $B_2=0$. 
The coefficient $B_2$ is calculated from the Taylor expansion as stated in Theorem \ref{th:cayley}: 
$$
B_2 = -\frac{(D+2E)^2(4(D^2-4)E^2 + 4D(D^2-3)E+D^4-2D^2-3)}{2|1+2DE + 4E^2|}.
$$
By assumption \refeq{eq:singular-conditions}, $D+2E \neq 0$ and $1+2DE + 4E^2 \neq 0$, so $B_2 =0$ is equivalent to 
$$
4(D^2-4)E^2 + 4D(D^2-3)E+D^4-2D^2-3=0,
$$ 
which is precisely the condition stated in \cite{F21}. 
In Figure \ref{fig:Period3}, two $3$-periodic trajectories are shown.
\begin{figure}[ht]
	\centering
\includegraphics[width=0.43\textwidth]{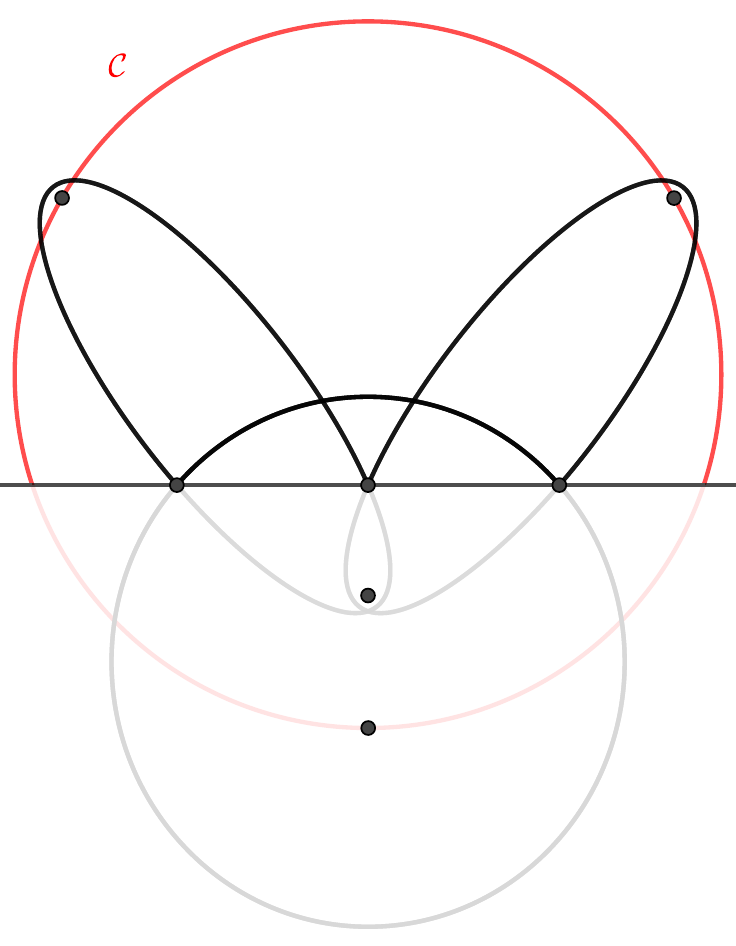}  
\hskip0.5cm
\includegraphics[width=0.43\textwidth]{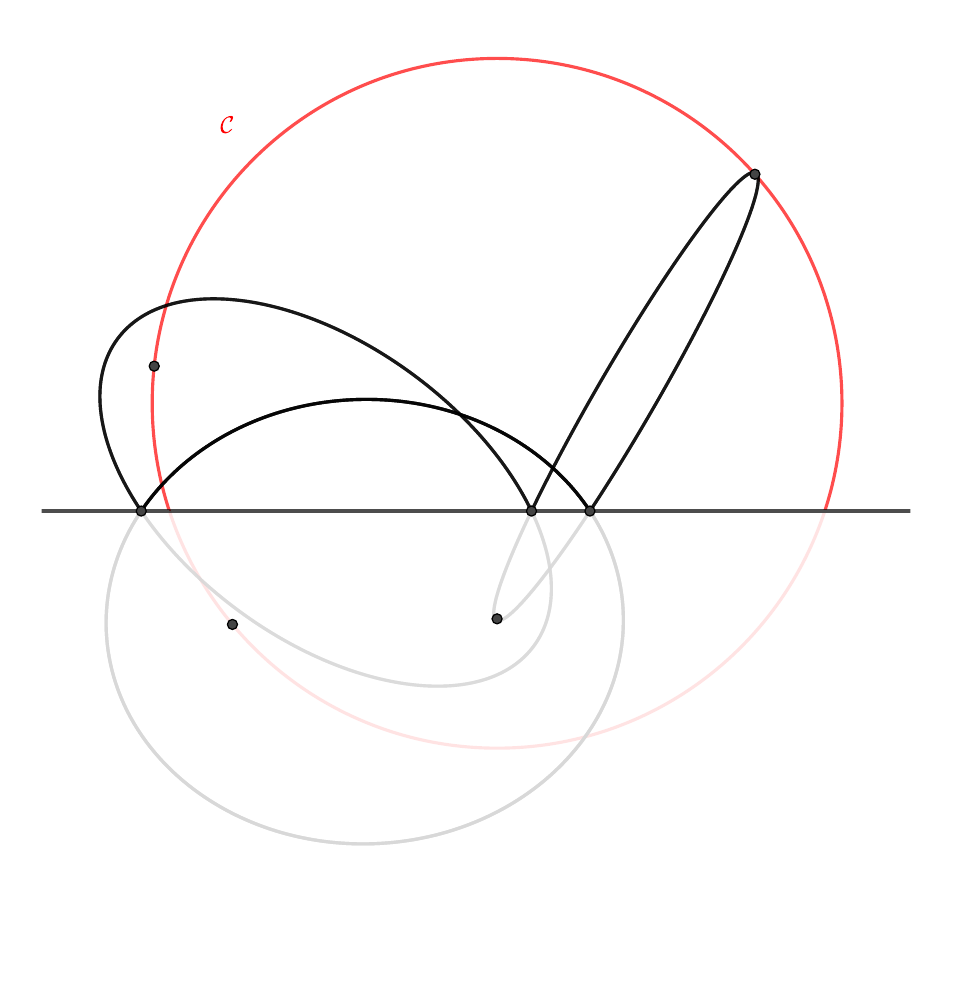}
\caption{Period 3 trajectories corresponding to $(E,D) = (-5/24,7/4).$
The trajectory on the left-hand side is symmetric with respect to the vertical axis.} 
\label{fig:Period3}
\end{figure}
\end{example}

\begin{example}[Period 4] \label{ex:Period4}
Theorem \ref{th:cayley} states that the analytic Cayley-type condition for $4$-periodic trajectories is $B_3=0$, which is equivalent to:
$$
(D^2+2DE-1)((D+2E)^2(D^2-4)-1)=0.
$$ 
Examples are shown in Figure \ref{fig:Period4}.
\begin{figure}[ht]
	\centering
\includegraphics[width=0.43\textwidth]{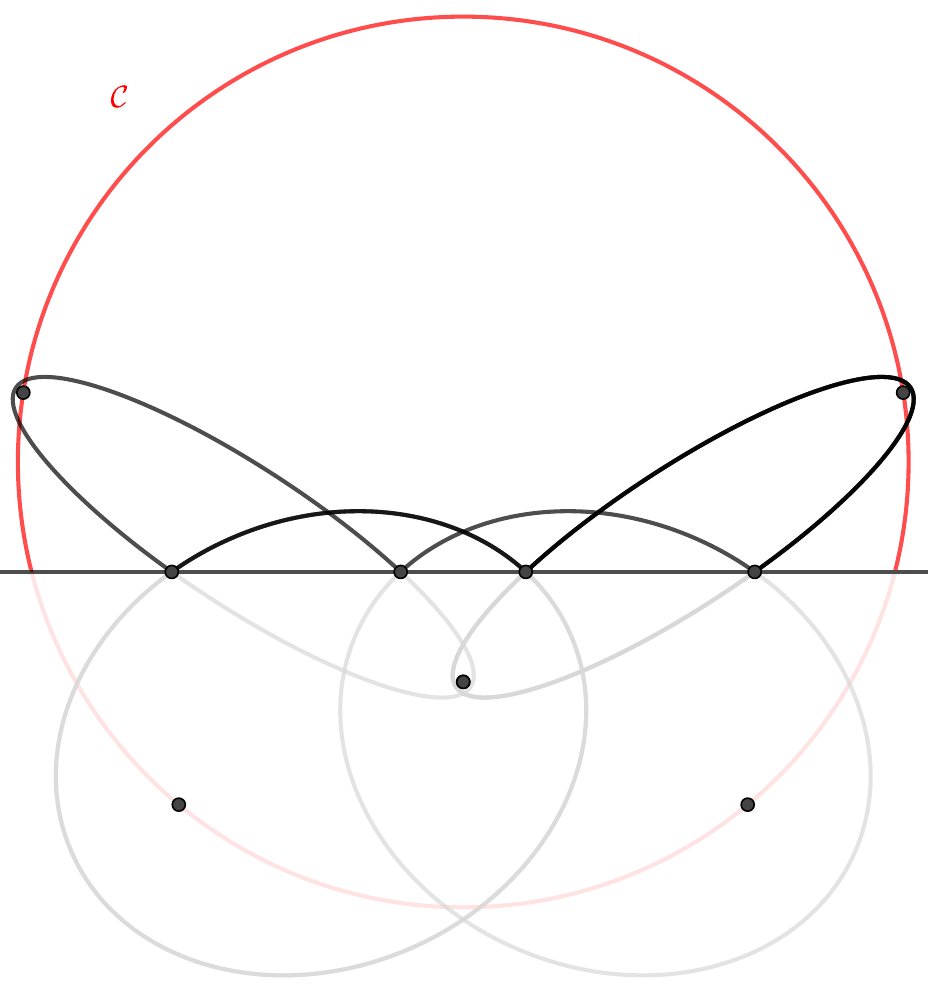}
\hskip0.5cm
 \includegraphics[width=0.43\textwidth]{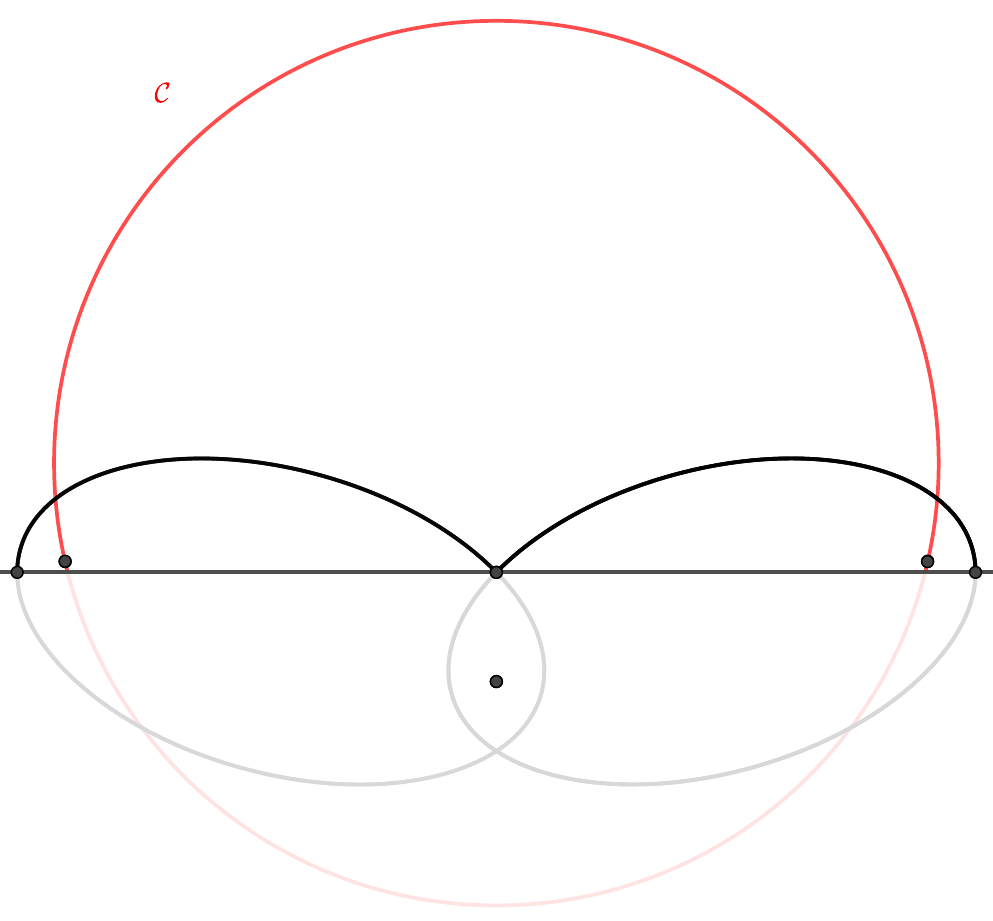}
\caption{Period 4 trajectories corresponding to $(E,D) = (-20/99,11/9)$. All $4$-periodic trajectories on this level set are symmetric with respect to the vertical axis. The trajectory in the right-hand side is consists of two arcs which are traced twice in opposite direction along the period, and they are orthogonal to the wall at the leftmost and the rightmost reflection points.} 
\label{fig:Period4}
\end{figure}
\end{example}

\begin{example}[Period 5]\label{ex:Period5}
The condition for a $5$-periodic trajectory is $B_2 B_4 - B_3^2=0$, which is equivalent to: 
\begin{align*}
0 =\ &D^{12}-6 D^{10}+3 D^8+60 D^6-169 D^4+42 D^2+5
+64 \left(D^2-4\right)^3 E^6
\\
&
+64 (D^2-4) \left(3 \left(D^2-7\right) D^2+52\right) D E^5 
+16 (D^2-4)\left(15 D^6-90 D^4+251 D^2+4\right) E^4 \\
&+32 (D^2-4) \left(5 D^6-25 D^4+71 D^2+13\right) D E^3  
\\
&
+4 \left(386 D^6-452 D^4-537 D^2+15 \left(D^2-8\right) D^8+52\right) E^2 \\
&+4 \left(3 D^{10}-21 D^8+46 D^6+22 D^4-257 D^2+47\right) D E.
\end{align*}
In Figure \ref{fig:Period5}, two such trajectories are shown.
\begin{figure}[ht]
	\centering
	\begin{tabular}{c c}
		\includegraphics[width=0.43\textwidth]{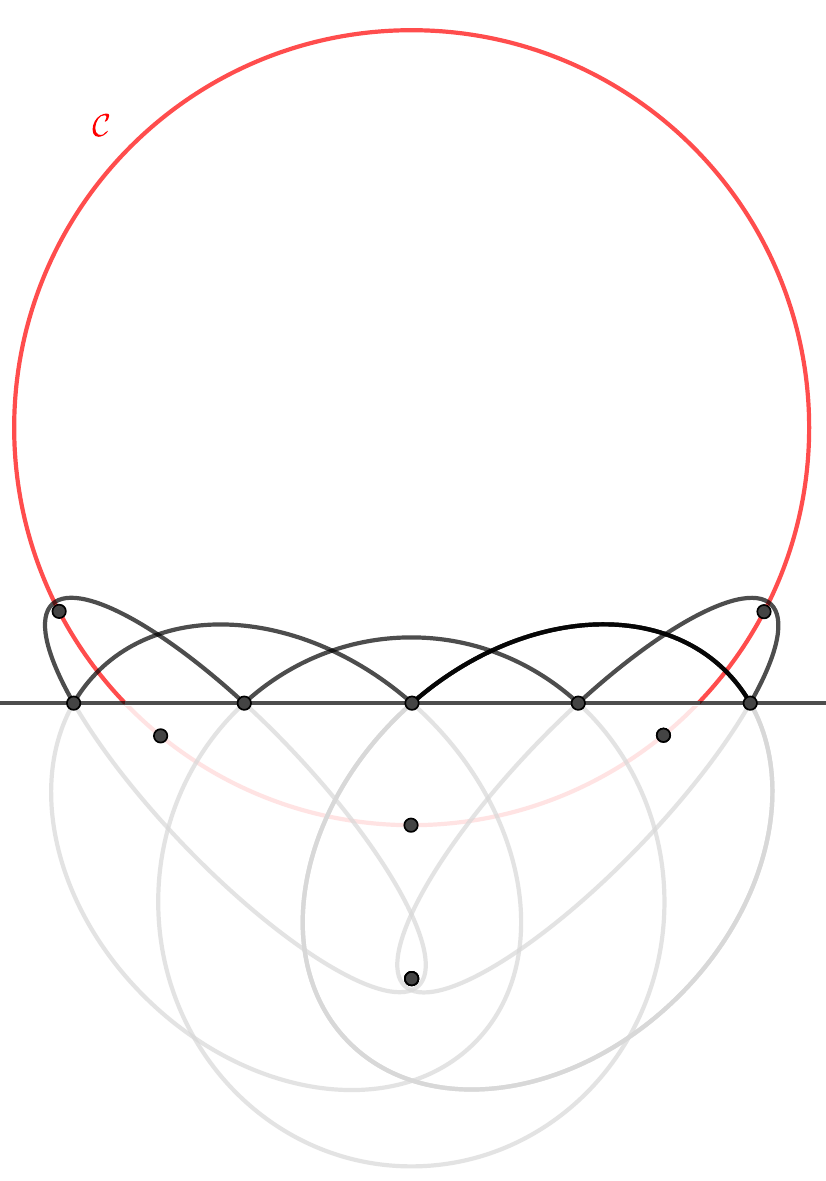} & \includegraphics[width=0.43\textwidth]{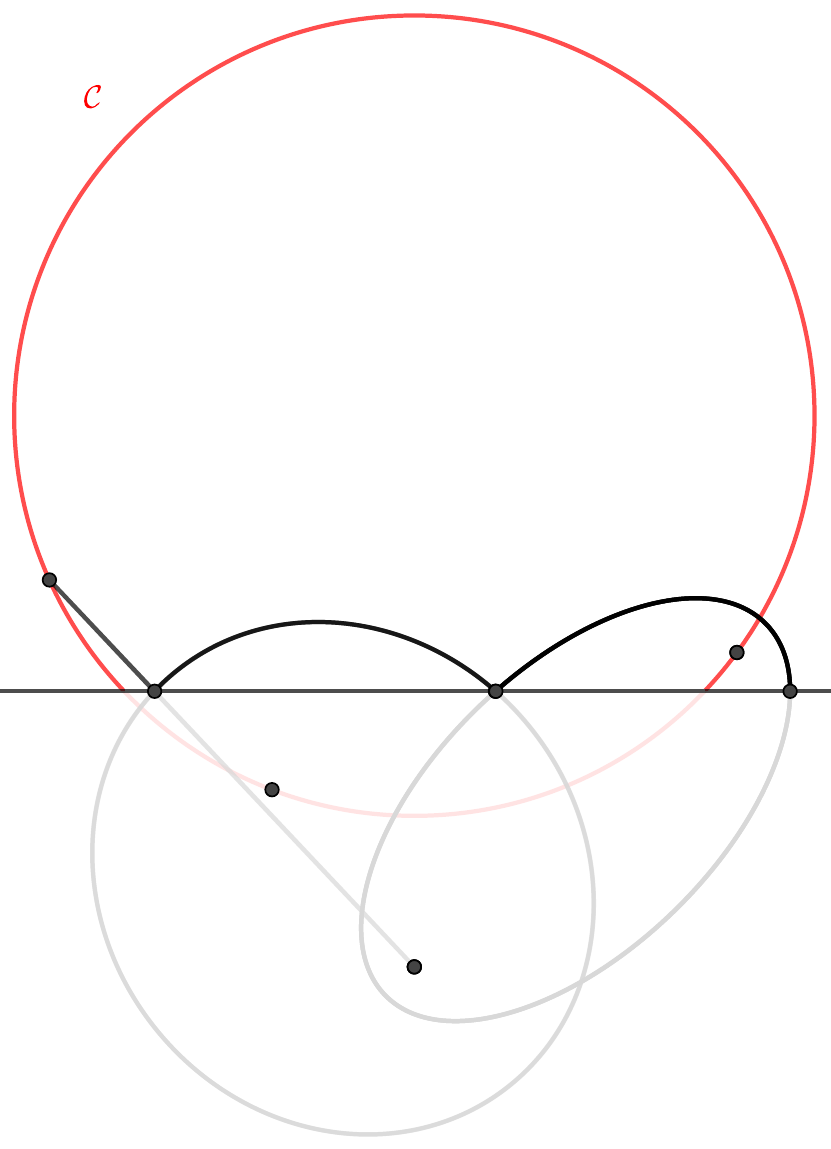}
	\end{tabular}
\caption{Period 5 trajectories corresponding to $E\approx-2/3$,  $D=\sqrt{3}$. The trajectory in the left-hand side is symmetric with respect to the vertical axis. In the right-hand side of the figure, the leftmost arc is close to the degenerate one.}
\label{fig:Period5}
\end{figure}
\end{example}

\begin{example}[Period 6]\label{ex:Period6}
The analytic condition $B_3 B_5 - B_4^2=0$ for $6$-periodic trajectories is equivalent to: 
\begin{align*}
0 &= \left[ 4(D^2-4)E^2 + 4D(D^2-3)E+D^4-2D^2-3 \right] \left[(D^2+2DE-1)^2-4(D+2E)^2 \right] \\
&\;\;\;\;\; \times \left[-1+(D^2-4)(D+2E)^2 ((3D^2-4)(D+2E)^2+16E(D+2E)+6) \right].
\end{align*}
The first factor in the above expression is the condition for $3$-periodic trajectories, so we find the solutions from the other two factors. Two such examples are shown in Figure \ref{fig:Period6}.
\begin{figure}[ht]
	\centering
\includegraphics[width=0.43\textwidth]{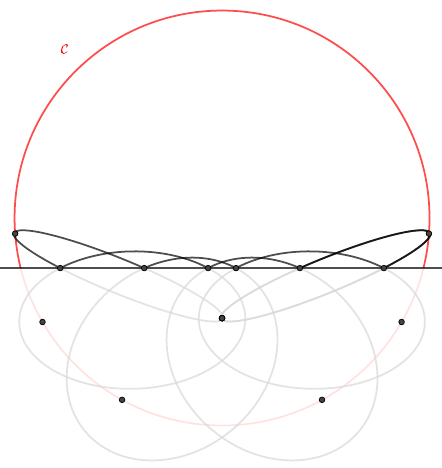} 
\hskip0.5cm
 \includegraphics[width=0.43\textwidth]{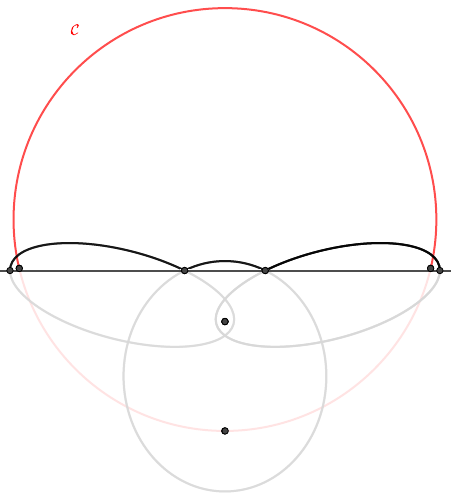}
\caption{Period 6 trajectories corresponding to $(E,D) = (-31/140,4/5)$. }
\label{fig:Period6}
\end{figure}
\end{example}

\begin{remark}
The trajectories in the right-hand sides of Figures \ref{fig:Period4} and \ref{fig:Period6} meet the wall at a right angle at its leftmost and rightmost points. These \emph{perpendicular trajectories} are fixed in the involution $j$ of \cite{F21} and occur when $A_2 = \frac{2E}{1\pm R}$.

\end{remark}

\section{Geometric properties of the Boltzmann system}\label{sec:geom}

In planar elliptical billiards, each trajectory has \emph{a caustic}, that is a curve which is touching all segments of that trajectory.
Moreover, \emph{the focal property} also holds: if a segment of a given trajectory contains a focus of the boundary, then then the next segment will containing the other focus.
In this section, we will prove that the trajectories of the Boltzmann system have caustics and that the focal property can also be formulated and proved in this case.

The first step is to establish that for any Kepler conic $\K$, there are particular confocal conics that are tangent to $\K$. See Figure \ref{fig:CausticIllustration}, left. 

\begin{proposition}\label{prop:kepler-caustic}
The Kepler conic $\K$ given by \eqref{eq:KeplerConic} is touching two unique conics with foci $\O(0,0)$ and $\F(0,2)$, which are given by:
\begin{equation}\label{eq:GeneralCaustic}
	\E_{\pm}: \quad \frac{x_1^2}{\left(\frac{R\pm1}{2E} \right)^2 -1} + \frac{(x_2-1)^2}{\left(\frac{R\pm1}{2E} \right)^2} =1.
\end{equation}
Moreover, we have:
\begin{itemize}
\item the points of tangency of $\K$ and $\E_{\pm}$ are
$B_{\pm} = \left(A_1\alpha_{\pm}, 2+(A_2-2E)\alpha_{\pm}\right)$, with
\begin{equation*}
	\alpha_{\pm}=
	\frac{(4E^2-(R\pm1)^2) (R(R\pm1)-2E(A_2-2E))}{2ER^2(4E^2-(R\pm1)^2)-8A_1^2E^3};
\end{equation*}
\item the slope of the joint tangent line to $\K$ and $\E_{\pm}$ at $B_{\pm}$ is:
\begin{equation*}
	m_{\pm} = \frac{A_1(R\pm1)(2E(A_2 - 2E)-(R\pm1)R)}{2A_1^2 E (R\pm1) - R(A_2 - 2E)(4E^2- (R\pm1)^2)};
\end{equation*}
\item the points $B_{\pm}$, $\F$ and the non-origin focus $F_2$ of $\K$ are collinear and they lie on the following line:
\begin{equation*}
	\L: \quad (2E-A_2)x_1 + A_1 x_2 = 2 A_1. 
\end{equation*}
\end{itemize}
\end{proposition}
\begin{proof} Follows by a straightforward calculation. 
\end{proof}

Now, Proposition \ref{prop:kepler-caustic} directly implies the existence of caustics for the Boltzmann system.

\begin{theorem}\label{th:caustics}
For each fixed pair $(E,D)$ satisfying the conditions \refeq{eq:singular-conditions} there are two unique conics $\E_+$ and $\E_-$ with foci $\O(0,0)$ and $\F(0,2)$ such that all arcs of each trajectory on the level set $X(E,D)$ are touching those two conics.
\end{theorem}

\begin{example}
The existence of caustics is illustrated in Figure \ref{fig:Caustic}, were $100$ iterations of the Boltzmann map for $(E,D) = (-7/24,7/4)$ are shown. 
The arcs of the trajectory have a hyperbolic caustic above the wall and the an elliptical caustic with tangencies both above and below the wall. 
\begin{figure}[ht]
		\centering
		\includegraphics[width=0.43\textwidth]{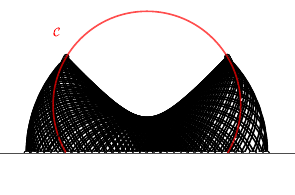} \hskip0.5cm \includegraphics[width=0.43\textwidth]{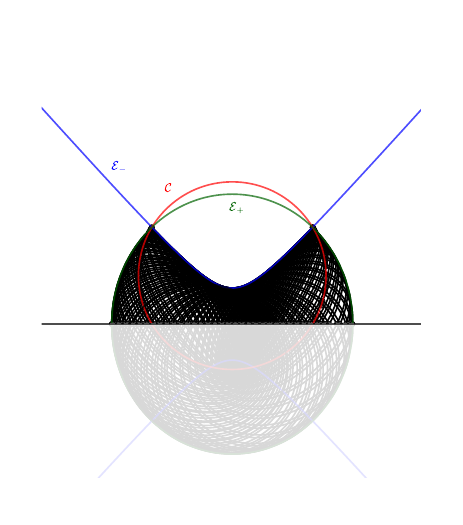}
		\caption{Caustics for the Boltzmann map. The right-hand figure shows the caustics $\E_\pm$ and the shaded full ellipses below the wall. }
		\label{fig:Caustic}
	\end{figure}
\end{example}

Next, we will show that the point which is symmetric to the centre with respect to the wall has the focusing property for the trajectories of the Boltzmann system, see Figure \ref{fig:focal}.

\begin{figure}[ht]
	\centering
	\includegraphics[width=0.5\textwidth]{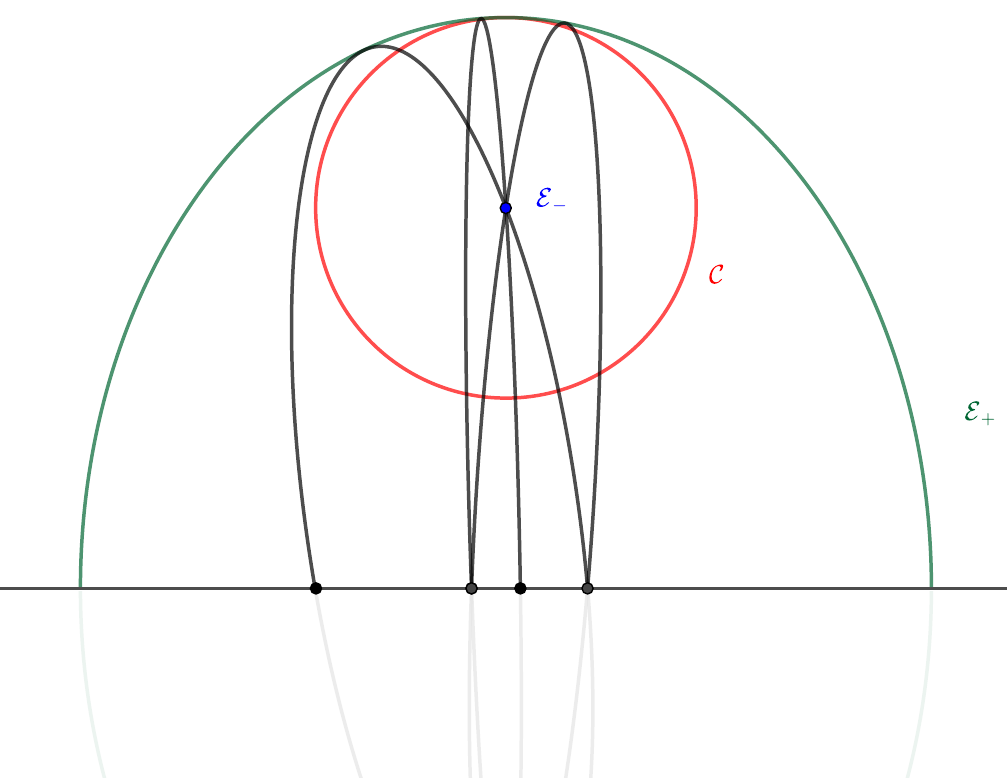} 
	\caption{Three consecutive arcs of a trajectory containing the point $\F(0,2)$, which is symmetric to the centre with respect to the wall.}
	\label{fig:focal}
\end{figure}

\begin{theorem}[Focal Reflection Property]\label{th:focal}
Suppose that an arc of a given trajectory of the Boltzmann systems contains the point $\F(0,2)$.
Then all arcs of that trajectory will also contain point $\F$. Furthermore, the trajectory asymptotically converges to the vertical axis $x_1=0$.
\end{theorem}

\begin{proof}
	The first part of this theorem follows from Theorem \ref{th:caustics}, when the caustic $\E_-$ given by \refeq{eq:GeneralCaustic} is degenerate, i.e.~$R-1=2E$, which implies that $D=2$.
	
	The corresponding Kepler ellipses have major axis $2a=-1/E$. 
	Denote by $x_0$ a point on the wall which belongs to one such ellipse $\K_0$,
	 and let $x_1$, $x_2$, $\ldots$ and $\K_1$, $\K_2$, \dots be the subsequent reflection points and arcs of the Boltzmann trajectory. 
The origin $\O$ is the focus of each elliptic arc $\K_i$ and we denote by $F_i$ the other focus. 
	
	By the caustic property, the trajectory starts at $x_0$, travels along the first Kepler ellipse $\K_0$, passes through the point $\F$ of the degenerate caustic $\E_1$ and intersects the wall at $x_1$. The ``string construction" applied to ellipses $\K_0$ and $\K_1$ implies $-1/E = |\pazocal{O}x_1| + |x_1 F_0|$ and $-1/E = |\pazocal{O} x_1| + |x_1 F_1|$ respectively. 
	Those equalities give $|x_1 F_0| = |x_1 F_1|$, so that we may think of $F_0$, $F_1$ as the intersection of $\C$ and a circle centred at $x_1$ with radius $R_{01}$. By the position of $x_1$, it is necessarily the case that the $x_2$-coordinate of $F_1$ is larger than the $x_2$-coordinate of $F_0$. 
	
	Repeating this process with $\K_1$ and $\K_2$ produces the next focus $F_2$ with larger $x_2$-coordinate than that of $F_1$. We find the $x_2$-coordinates of the $F_i$ are monotonically increasing and are bounded above by the top of the circle $\C$ at coordinates $(0,2-1/E)$. As the $F_i$ approach this point, the defining components of each Kepler ellipse $(x,A_1,A_2)$ approach $(0,0,-1)$, respectively, which corresponds to the vertical axis $x_1=0$.
\end{proof}

\section{The phase space}\label{sec:phase}

In this section, we will analyse the phase space of the Boltzmann system. 
While the discussion in \cite{F21} assumes complex values of $D$, $E$, we will consider only real values to correspond to the real motion in the Kepler problem. 
Moreover, we will consider only the part of the phase space containing bounded trajectories, i.e.~when the arcs of the trajectories are ellipses.
We note that the last assumption implies that each trajectory is bounded and has infinitely many reflections.
Boundedness is important for us, since then the level sets and the isoenergy manifolds in the phase space will be compact, which allows us to use the Fomenko invariants for the topological characterization.

Subsection \ref{sec:bifurcation} contains a brief review of the bifurcation set for the Boltzmann system, based on \cite{F21}.
Then, in Subsection \ref{sec:singular}, we provide the analysis of the motion on the singular level sets.
In Subsection \ref{sec:Fomenko}, we present the topological description of the compact isoenergy manifolds for the Boltzmann system, using Fomenko graphs.

\subsection{The bifurcation set}\label{sec:bifurcation}

The bifurcation set for the Boltzmann system can be represented in $(E,D)$-plane, with restrictions on $D$ and $E$ to produce real motion when the arcs of trajectories are ellipses or degenerate to straight segments. 
In particular, these restrictions determine an infinite region in the plane bounded by the curves $1+2DE + 4E^2=0$, $E=0$, $D+2E=0$, $D=2$, as shown in the right-hand side of Figure \ref{fig:CausticIllustration} \cite{F21}. We call this region $\pazocal{R}$.

As explained in Theorem \ref{th:caustics}, the caustics $\E_{\pm}$  are only dependent upon the values of $(E,D)$. 
The curve $\E_-$ is an ellipse for all $(E,D) \in \pazocal{R}$. However, the curve $\E_+$ is a hyperbola for $(E,D) \in \pazocal{R}$ with $D<2$, an ellipse for $(E,D) \in \pazocal{R}$ with $D>2$, and degenerate consisting of the two points $\O (0,0)$ and $\F(0,2)$ when $(E,D) \in \pazocal{R}$ with $D=2$ and $E>-\frac{1}{2}$.
The left-hand side of Figure \ref{fig:CausticIllustration}, shows an arc of a trajectory corresponding the level set $X(E,D)$.
\begin{figure}[ht]
	\centering
	\begin{tabular}{c c}
		\includegraphics[width=0.45\textwidth]{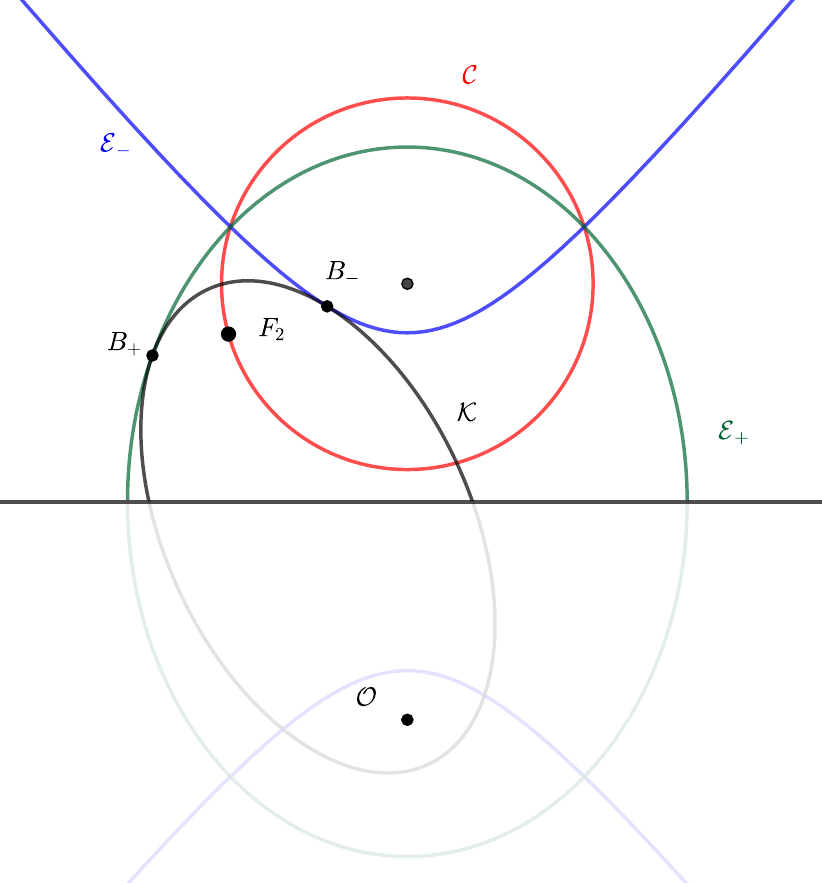} &		\includegraphics[width=0.45\textwidth]{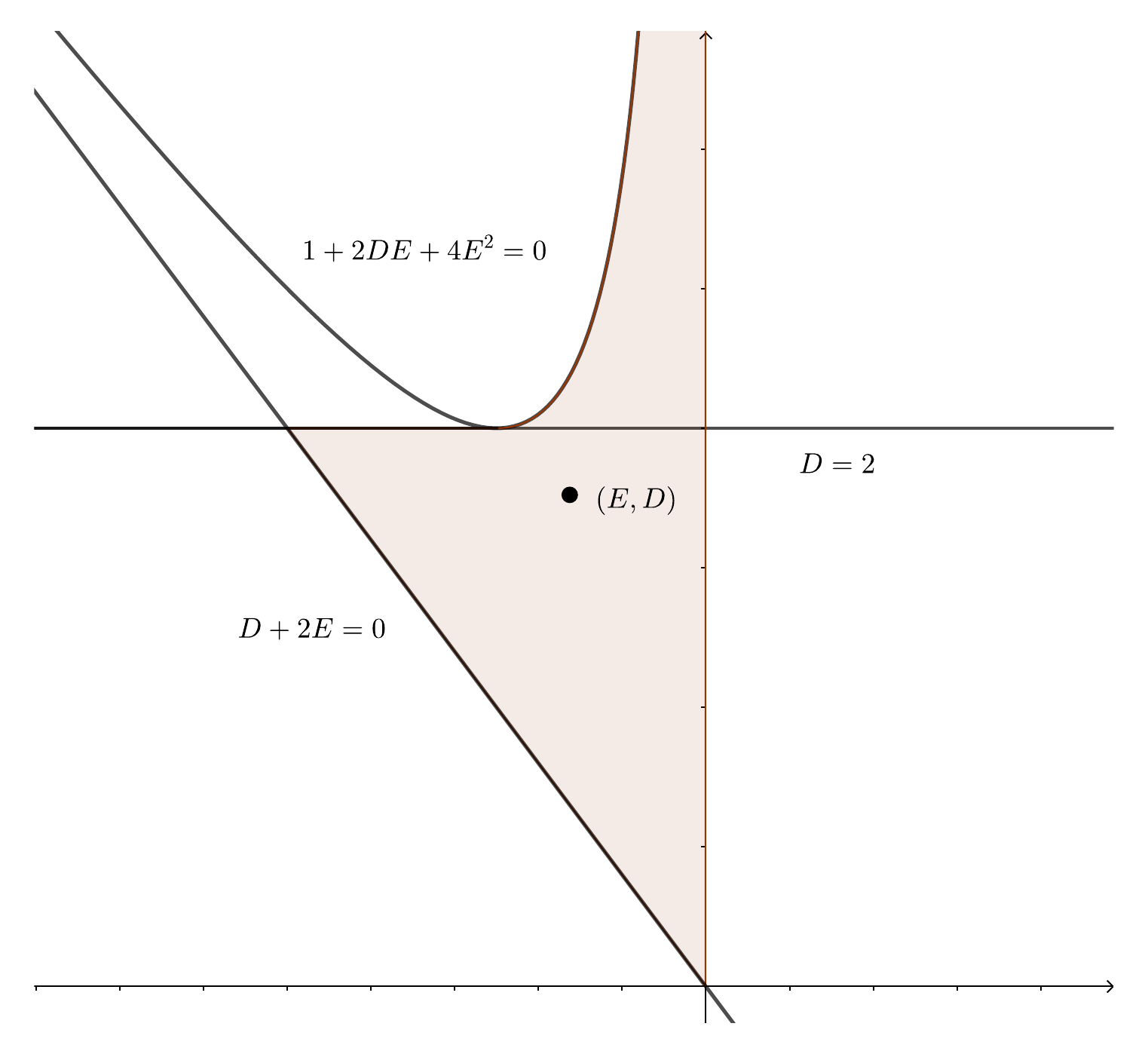}
	\end{tabular}
	\caption{The Kepler conic $\K$, caustics $\E_{\pm}$, foci, and the points $B_{\pm}$ (left); the $ED$-plane and the boundary curves of the shaded region $\pazocal{R}$ (right) which correspond to real motion. }
	\label{fig:CausticIllustration}
\end{figure}

\subsection{Singular level sets}\label{sec:singular}

The singular level sets of the Boltzmann system are placed on the boundary of the region $\pazocal{R}$ and within $\pazocal{R}$ on the line $D=2$. The next theorem describes the dynamics on those level sets.

\begin{theorem}\label{th:singular-level-sets}
The singular level sets in the phase space of the Boltzmann system consist of:
\begin{itemize}
	\item A single closed orbit corresponding to the limiting motion on the wall along the minor axis of the ellipse $\E_+$, for each $(E,D)$ such that $D+2E=0$ and $0<D<2$;
	\item A single closed orbit, corresponding to a $2$-periodic trajectory on an ellipse whose minor axis is placed along the wall, for each $(E,D)$ such that $1+2DE+4E^2=0$ and $D>2$;
	\item A single closed orbit corresponding to a periodic trajectory  lying on the $x_2$-axis and bounded by the point $(0,-1/E)$ when $D=2$ and $-1<E<-1/2$;
	\item A closed orbit and a separatrix when $D=2$ and $-1/2<E<0$. The closed orbit corresponds to a periodic trajectory lying on $x_2$-axis and bounded by the point $\F(0,2)$ and the wall, and the separatrix contains the trajectories with elliptic arcs that contain the point $\F(0,2)$.
\end{itemize}
\end{theorem}

\begin{proof}
	The singular level sets correspond to the values $E$, $D$ which do not satisfy the inequalities \refeq{eq:singular-conditions}.
	We consider separately each of the three possible cases.

\paragraph*{First, we assume $D+2E=0$.}

By setting $x_2 =1$, we can find the $x_1$-coordinates of the reflection points of the particle with the wall:
$$
x_1^\pm = \frac{-A_1(A_2+D) \pm L \sqrt{D+2E}}{1-A_1^2},
$$
for $A_1^2 \neq 1$. 
From there,   the Kepler conics will be hyperbolas tangent to the wall from above when $E>0$; or parabolas and ellipses tangent to the wall from below for $E \leq 0$ if and only if $\Delta =0$.

Thus, the limiting motion for the Boltzmann system, whenever $E\le0$, will be along the minor axis of the limiting caustic $\E_+$, see Figure \ref{fig:Fomenko1a}.
\begin{figure}[ht]
	\centering
	\includegraphics[width=0.5\textwidth]{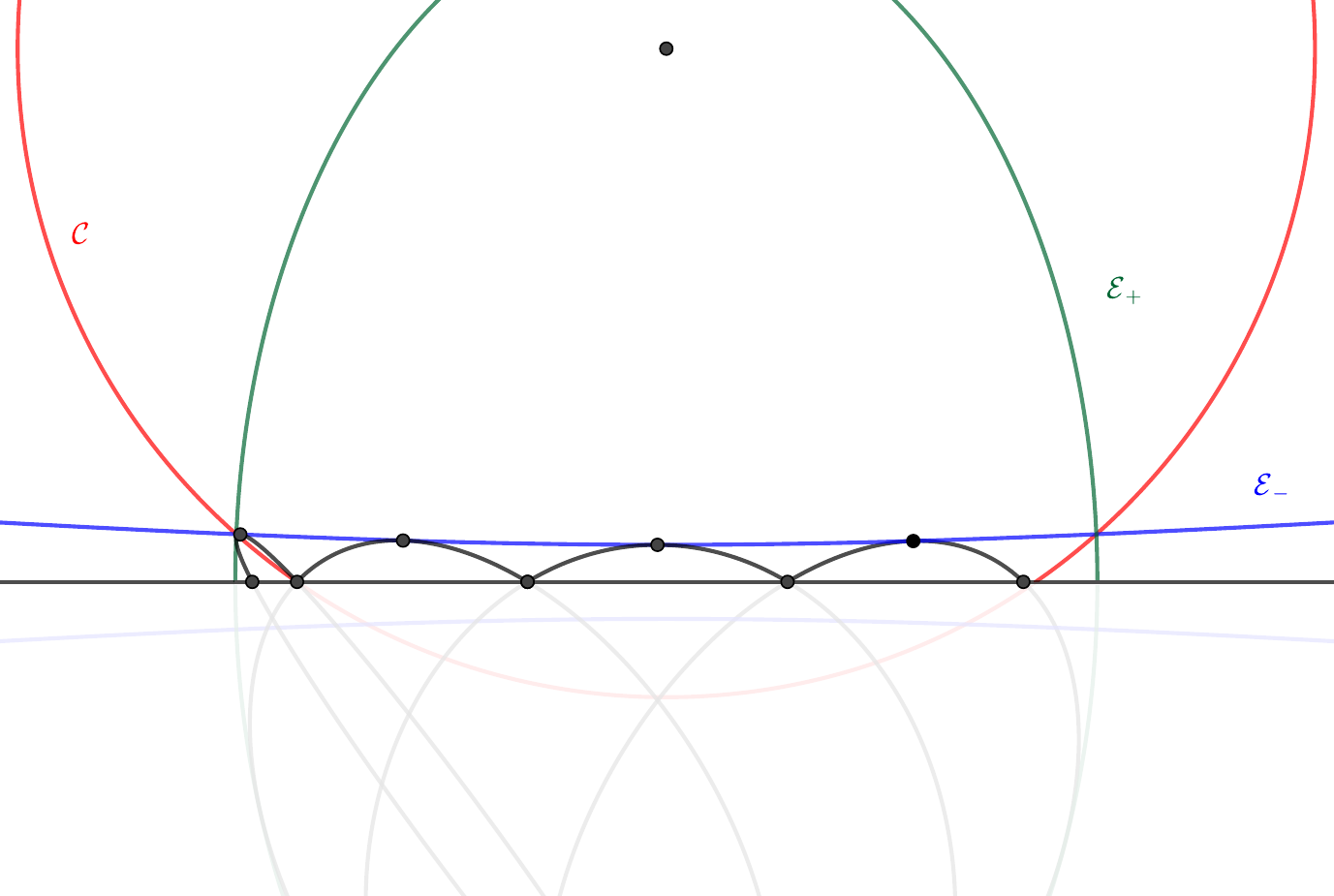} 
	\caption{Four iterations of the Boltzmann map with $(E,D)$ near $D+2E=0$ with $-1<E<-\frac{1}{2}$. }
	\label{fig:Fomenko1a}
\end{figure}

\paragraph*{Second, we assume $1+2DE+4E^2=0$.}
Under this assumption, the Kepler conics are ellipses orthogonal to the wall at both intersection points, i.e.~their minor axes lie on the wall $x_2 =1$ and major axes lie on the vertical coordinate axis, $x_1=0$. See Figure \ref{fig:Period2FamilyFig}.

We can derive this condition using the involutions $i$, $j$, see \eqref{eq:involutions}. First, a necessary condition for such a 2-periodic ellipse is $x_1^+ = -x_1^-$, or equivalently, $i(x_1,A_1, A_2) = (-x_1, A_1, A_2)$. This means
\begin{equation}
	-\frac{2(A_2+D)A_1}{1-A_1^2} -x_1 = -x_1  \iff A_1 =0 \text{ or } A_2 = -D,
	\label{eq:iCondition}
\end{equation}
which correspond to conics whose intercepts with the wall are symmetric about $x_1=0$. To further correspond to a 2-periodic trajectory, we seek fixed points of $j$:
\begin{equation}
	j(x_1,A_1,A_2) = (x_1,A_1,A_2) \iff A_1 = x_1(2E-A_2).
	\label{eq:jCondition}
\end{equation}
Satisfying equations (\ref{eq:iCondition}) and (\ref{eq:jCondition}) leads to several possibilities. \\

\emph{Case 1: Suppose $A_1=0$ and $x_1 \neq 0$.} Then $A_2 = 2E$, and the second focus of the conic (\ref{eq:KeplerConic}) is $F_2 = (0,2)$, and the equation of the conic simplifies to
\begin{equation}
	\frac{x_1^2}{-\frac{D}{2E}-2} + \frac{(x_2-1)^2}{\frac{1}{4E^2}}=1.
	\label{eq:Case1a}
\end{equation}
By equation (3) of \cite{F21}, these conditions also mean $
R^2=0 \iff 1+2DE + 4E^2 =0$. This matches Felder's 2-periodic condition and is consistent with the geometric description that the second focus $F_2$ in the Boltzmann system lies on a circle of radius $R/|E|$ centred at $\F(0,2)$. In turn, equation (\ref{eq:Case1a}) becomes  
\begin{equation}
	\frac{x_1^2}{\frac{1}{4E^2}-1} + \frac{(x_2-1)^2}{\frac{1}{4E^2}}=1.
	\label{eq:Case1b}
\end{equation}
This equation represents an ellipse in the $(x_1,x_2)$ plane for $-\frac{1}{2} < E < \frac{1}{2}$ and $E \neq 0$. Moreover, in the $(E,D)$-plane, the curve $1+2DE+4E^2=0$ is a hyperbola with asymptotes $D+2E=0$ and $E=0$, and branches lying in the second and fourth quadrants. Further assuming $D+2E>0$ to ensure two distinct intersection points of the ellipse with the wall, this forces the pair $(E,D)$ to lie on the branch in the second quadrant. Therefore we have a 1-parameter family of ellipses corresponding to 2-periodic trajectories given by equation (\ref{eq:Case1b}) for $-\frac{1}{2}<E<0$. This family of ellipses approaches a degenerate ellipse (or segment connecting the foci) as $E \to -\frac{1}{2}^+$ and approaches an ellipse of unbounded major and minor axes as $E \to 0^-$. 
\begin{figure}[ht]
	\centering
	\includegraphics[width=0.5\textwidth]{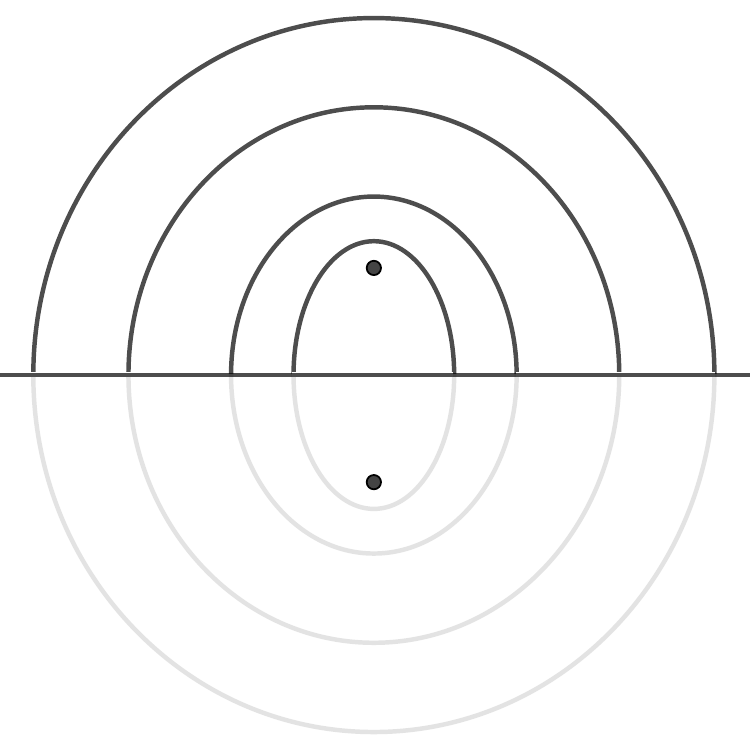}
	\caption{A family of 2-periodic ellipses given by equation (\ref{eq:Case1b}) and varying $E$ corresponding to different 2-periodic trajectories. }
	\label{fig:Period2FamilyFig}
\end{figure}

\emph{Case 2: Suppose $A_1=0$ and $x_1=0$.} The Kepler conic equation (\ref{eq:KeplerConic}) reduces to 
\begin{equation}
	x_1^2+x_2^2 = (2A_2+D-A_2x_2)^2
	\label{eq:Case2a}
\end{equation}
and must pass through the point $(0,1)$, which implies $A_2 = -D \pm 1$. However, this reduces equation (3) of \cite{F21} to 
$$
(D+2E)(D\pm2)=0,
$$
with $D$ having the opposite sign of $A_2$. By assuming $D+2E>0$, both possibilities turn equation (\ref{eq:Case2a}) into $x_1=0$. Thus the only such conic is the degenerate conic $x_1=0$. Dynamically, this can be seen as the particle repeatedly bouncing directly up and down with no component of motion to the left or right. \\

\emph{Case 3: Suppose $A_2 = -D$.} Since $L^2=D+2A_2 = -D$, we have $D <0$, which does not belong to the region $\pazocal{R}$.

\paragraph*{Third, assume $D^2=4$.} 
Since $D$ is positive within the bifurcation set $\pazocal{R}$, we have $D=2$. 
In the Boltzmann system, the Kepler conic $\K$ is an ellipse whose foci are $F_1 = (0,0)$ and $F_2 = (A_1/E,A_2/E)$, and whose major axis has length $1/|E|$. The second focus $F_2$ lies on a fixed circle $\C$ of radius $R/|E|$ centred at $\F(0,2)$. The degenerate Kepler ellipses will occur when, as $F_2$ varies along $\C$, the length of the minor axis approaches 0, which occurs if and only if $F_2$ is a distance $1/|E|$ from the origin. Thus we seek solutions to the system
$$
x_1^2 + (x_2-2)^2 = \frac{1+2DE+4E^2}{E^2} \quad \text{ and } x_1^2+x_2^2 = \frac{1}{E^2}.
$$
The solutions are $(x_1,x_2) = \left(\pm \frac{\sqrt{4-D^2}}{2E},-\frac{D}{2E} \right).$ Thus the line $x_2=-\frac{D}{2E}$ is a line which can intersect $\K$ in zero, one, or two points; depending on the location of $F_2$ relative to this line, the Kepler conic (\ref{eq:KeplerConic}) will be an ellipse, hyperbola, or degenerate. See Figure \ref{fig:DTransitions}.
\begin{figure}[ht]
	\centering
	\begin{tabular}{c c c}
		\includegraphics[width=0.3\textwidth]{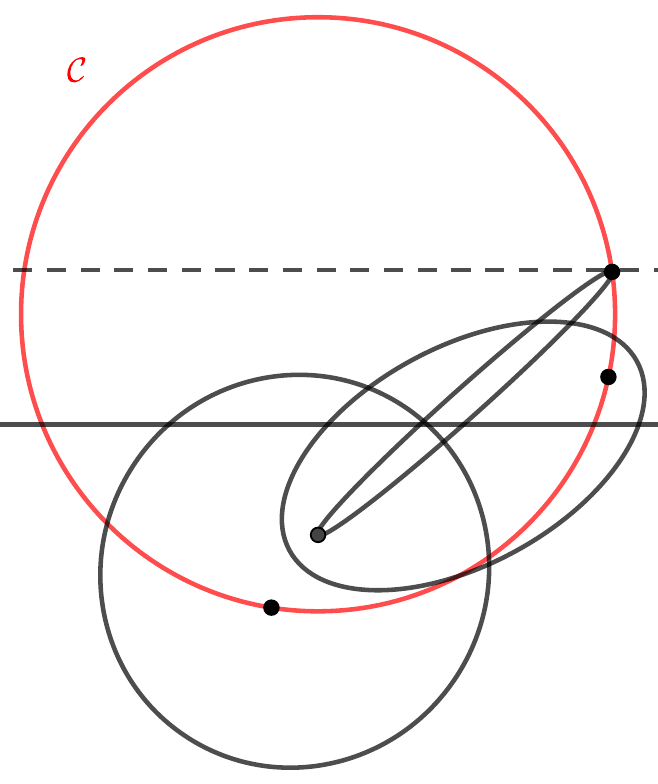} & \includegraphics[width=0.3\textwidth]{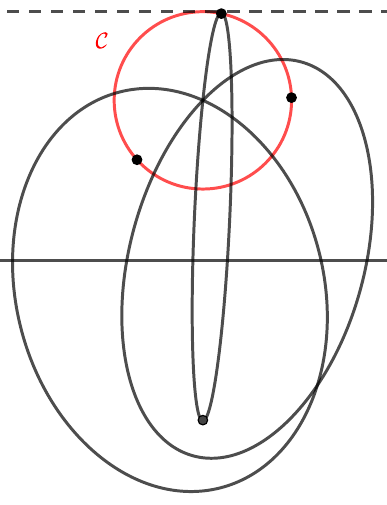} & \includegraphics[width=0.3\textwidth]{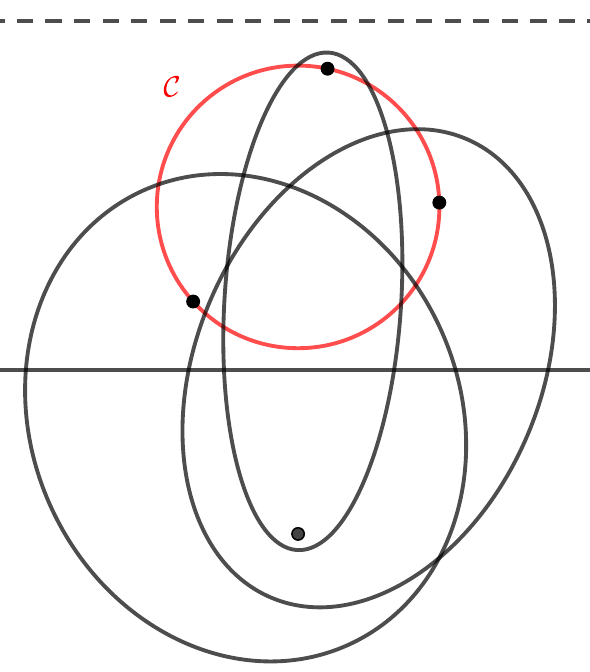}
	\end{tabular}
	\caption{Several Kepler ellipses corresponding to $D<2$ (left); $D=2$ (center); $D>2$ (right) and $-1/2<E<0$. In each figure, the dashed line is $x_2 = -\frac{D}{2E}$, the solid line the wall $x_2=1$. }
	\label{fig:DTransitions}
\end{figure}

The condition $D=2$ corresponds to the case when the line $x_2 = -\frac{D}{2E}$ is tangent to the circle $\C$ at the point $\left(0, -\frac{D}{2E} \right)$. As such, all points except the point of tangency are allowable locations for $F_2$ and the Kepler conic is an ellipse. 
\end{proof}

\subsection{Topological description of isoenergy manifolds}\label{sec:Fomenko}

In this section we give a topological description of the isoenergy manifolds for the Boltzmann system using the topological invariants developed by Fomenko and his school, see \cites{BMF1990, BF2004, BBM2010} and references therein.
Those invariants can be used for $3$-dimensional submanifolds of the phase space of integrable systems with two degrees of freedom.
The Liouville folitation of such submanifolds is represented by a graph, which is obtained by shrinking each leaf of the foliation to a point.
Thus, the smooth families of Liouville tori will create edges which connect together at vertices corresponding to the singular leaves. 
Each type of the singular leaf corresponds to a letter-atom.
To complete the topological description, each edge and some subgraphs are marked with rational and natural numbers.
A detailed account of those invariants, together with theoretical bakcground and examples can be found in \cite{BF2004}.

The Fomenko graphs were extensively used for studying the topology of integrable billiards: elliptical ones \cites{DR2009, DR2010}, within the domains bounded by confocal parabolas \cite{Fokicheva2014}, with Hooke's potential \cite{R2015}, in the Minkowski plane and on ellipsoids and the hyperboloid in the Minkowski space \cites{DR2017,DGR2022}, non-convex billiards \cite{Ved2019}, billiards with slipping \cite{FVZ2021}, and broader classes of systems and their bifurcations \cites{SRK2005,VK2018, FV2019, PRK2018,PRK2021}.
For the larger body of works on the topic, see also the references therein.

\begin{theorem}\label{th:fomenko}
The subsets of the phase space for the Boltzmann system corresponding to fixed negative values of $E$ are compact $3$-di\-men\-sio\-nal manifolds which are represented by the Fomenko graphs as shown in Figure \ref{FomenkoGraphs}.
\end{theorem}

\begin{figure}[ht]
	\centering
	\begin{tikzpicture}[>=stealth']
		\tikzset{vertex/.style = {shape=circle,draw,minimum size=1.5em}}
		\node[vertex] (aa) at  (2.5,1.5) {A};
		\node[vertex] (ff) at  (2.5,-1) {A};
		\node[vertex] (bb) at  (7.5,1.5) {$\text{A}^*$};
		\node[vertex] (ee) at  (7.5,4) {A};
		\node[vertex] (jj) at  (7.5,-1) {A};
		\path[->] 
		(ff) edge node[right]  {\footnotesize $\varepsilon=1$} node[left] {\footnotesize $r=0$} (aa) 
		(bb) edge node[right]  {\footnotesize $\varepsilon=1$} node[left] {\footnotesize $r=0$} (ee) 
		(bb) edge node[right]  {\footnotesize $\varepsilon=1$} node[left] {\footnotesize $r=0$} (jj); 
		\draw[dashed] (bb) circle (20pt) node[right=20pt]{\footnotesize$n=0$};
		\draw (0,-3) -- (10,-3);
		\foreach \x in {0,5,10}
		\draw[shift={(\x,-3)},color=black] (0pt,3pt) -- (0pt,-3pt);
		\node[below] at (0,-3) {\footnotesize $E=-1$};
		\node[below] at (5,-3) {\footnotesize $E=-\frac{1}{2}$};
		\node[below] at (10,-3) {\footnotesize $E=0$};
		\draw (-1,-2) -- (-1,5);
		\foreach \y in {-1,1.5,4}
		\draw[shift={(-1,\y)},color=black] (-3pt,0pt) -- (3pt,0pt);
		\node[left] at (-1,-1) {\footnotesize $D+2E=0$};
		\node[left] at (-1,1.5) {\footnotesize $D=2$};
		\node[left] at (-1,4) {\footnotesize $1+2DE+4E^2=0$};
	\end{tikzpicture}
	\caption{Fomenko graphs corresponding to the isoenergy manifolds for the Boltzmann system with $E<-\frac{1}{2}$ (left) and for $-\frac{1}{2}<E<0$ (right). }
	\label{FomenkoGraphs}
\end{figure}
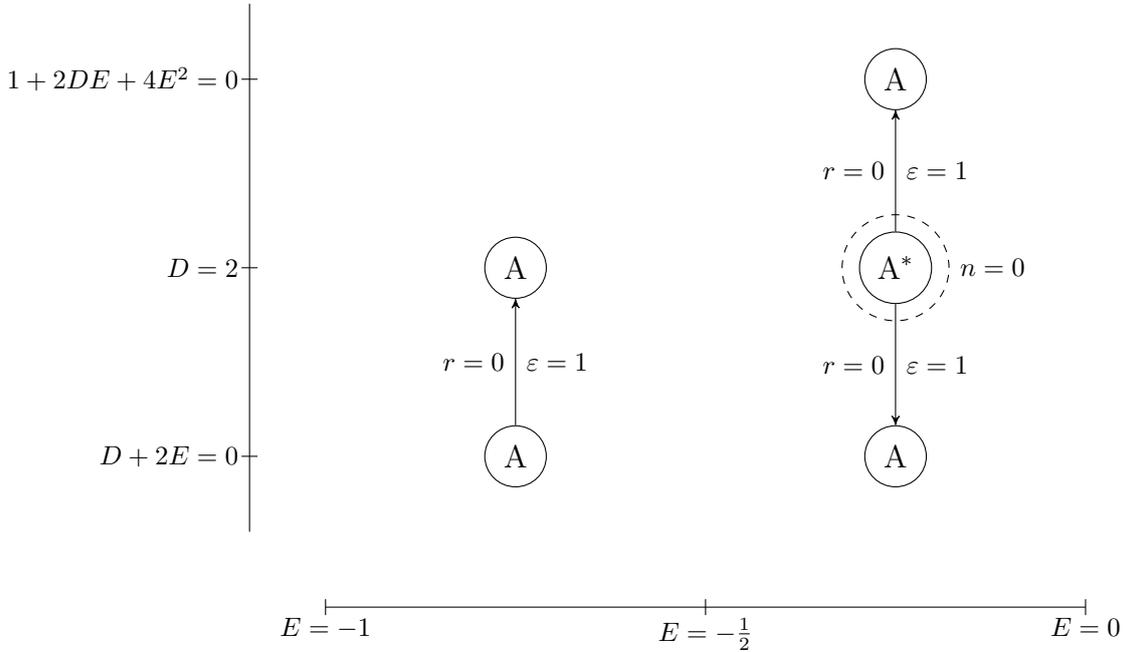

\begin{proof}
In order to determine the Fomenko graphs and the corresponding numerical invariants, we will analyse behaviour near the singular level sets.
	
\paragraph*{First, we consider the case $-1<E<-\frac{1}{2}$.}
For any pair $(E,D)$, such that $-1<E<-\frac12$ and $-2E<D<2$, the corresponding level set is a single Liouville torus, thus one edge of the Fomenko graph corresponds to those tori when $E$ is fixed and $D$ varies.
According to Theorem \ref{th:singular-level-sets},
each level set corresponding to $D\in\{2E,2\}$ consists of a single closed orbit, i.e.~the $A$-atom of the Fomenko graph.

Consider the nature of the Boltzmann trajectories near each of the boundary components. Near $D+2E=0$, the two branches of the caustic $\E_-$ are near the wall from above and below, keeping the arcs of the Kepler ellipses trapped vertically between the wall and $\E_-$, and horizontally within $\E_+$. These trajectories limit to the motion along the wall, between the vertices of the minor axis of $\E_+$, see Figure \ref{fig:Fomenko1a}. 

Near the upper bound $D=2$, the length of the minor axis of $\E_+$ shrinks to $0$.
The Kepler conics are bounded vertically between the wall and $\E_-$, and horizontally between the shrinking arcs of $\E_+$. 
These trajectories limit to a simple up-and-down 2-periodic trajectory between the wall at $(0,1)$ and the lowest point of $\C$, which has coordinates $(0,2-R)$. See Figure \ref{fig:Fomenko1b}. 

In order to calculate the numerical invariants, we need to choose two admissible bases on a Liouville torus corresponding to a point on the edge of the Fomenko graph. 
Each of those two bases is chosen accordingly to one of the singular level sets corresponding to the endpoints of the edge of the graph, and then the numerical invariants are calculated from the matrix of the transformation which maps one basis to the other one.
For details, see \cite{BF2004}.

In this case, one admissible basis, taken according to the singular level set with $D+2E=0$, can be chosen so that it consists of a preimage of a segment orthogonal on the wall and the segment placed on the wall.
The other admissible basis, taken according to the singular level set with $D=2$, can be chosen to consist of the preimages of the same segments, but in the reversed order.

\begin{figure}[ht]
\centering
\includegraphics[width=0.5\textwidth]{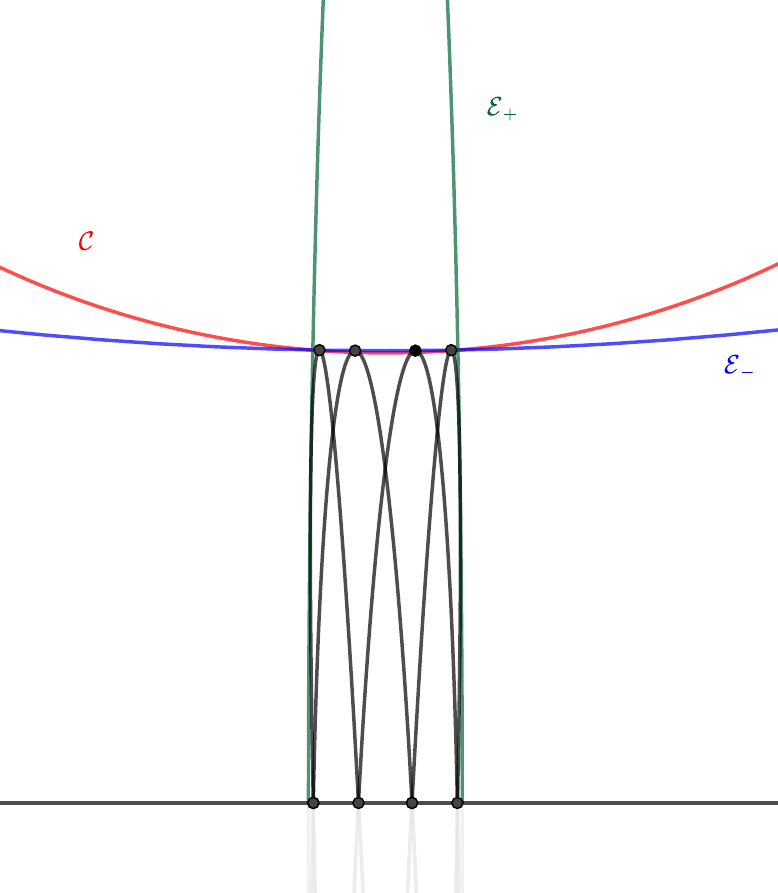} 
\caption{Four iterations of the Boltzmann map with $(E,D)$ near $D=2$ with $E<-\frac{1}{2}$. }
\label{fig:Fomenko1b}
\end{figure}

\paragraph*{Second, we analyse the case when $-\frac{1}{2}< E <0$.}
For any pair $(E,D)$, such that $-1/2<E<0$ and $D+2E>0$, $D\neq2$, $1+2DE+4E^2>0$, the corresponding level set is a single Liouville torus, thus the Fomenko graph has two edges: each one connecting the singular level set corresponding to $D=2$ with the two remaining singular level sets.
According to Theorem \ref{th:singular-level-sets},
that level set corresponds to the $A^*$-atom, while the two other level sets correspond to the $A$-atoms of the Fomenko graph.

Near the lower bound $D+2E=0$, the behaviour is the same as described in the previous case and shown in Figure \ref{fig:Fomenko1a}, thus the admissible basis can be chosen as in the previous case.
Near the boundary $D=2$, the caustic $\E_-$ narrows around the $x_2$-axis, as shown in the left of Figure \ref{fig:Fomenko2}.
\begin{figure}[ht]
	\centering
	\begin{tabular}{c c}
		\includegraphics[width=0.47\textwidth]{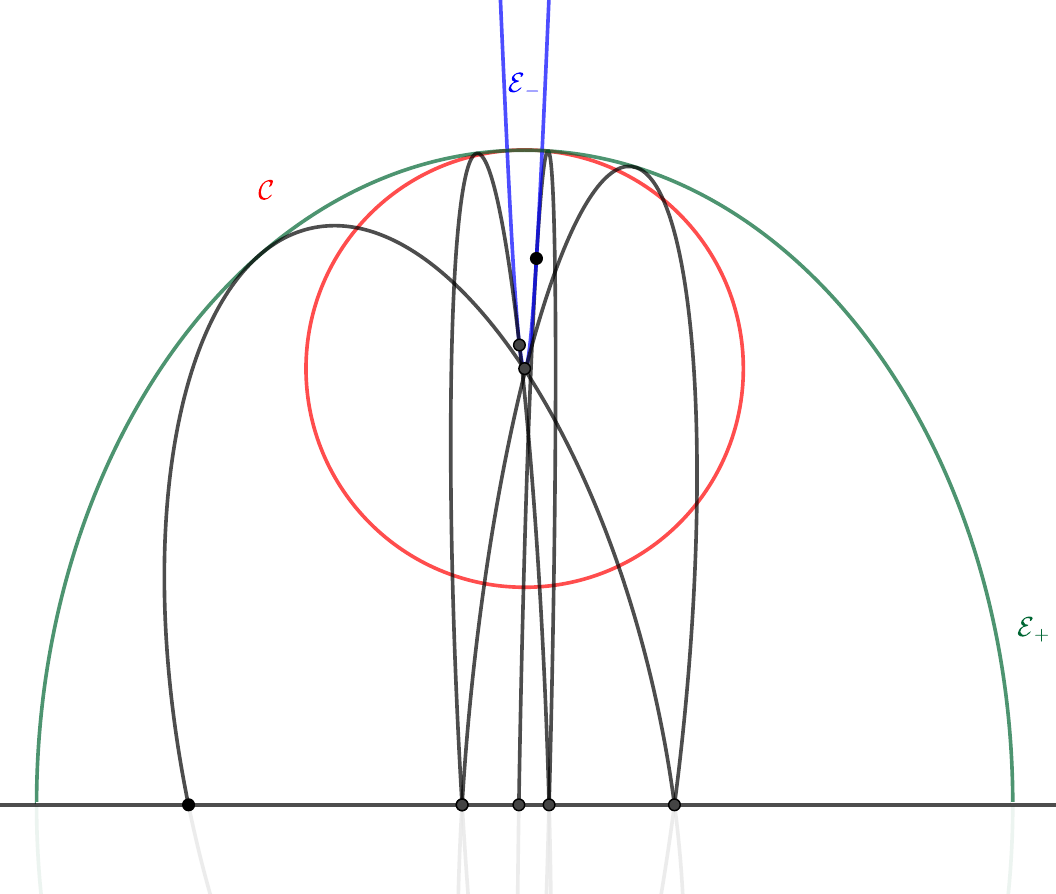} & \includegraphics[width=0.47\textwidth]{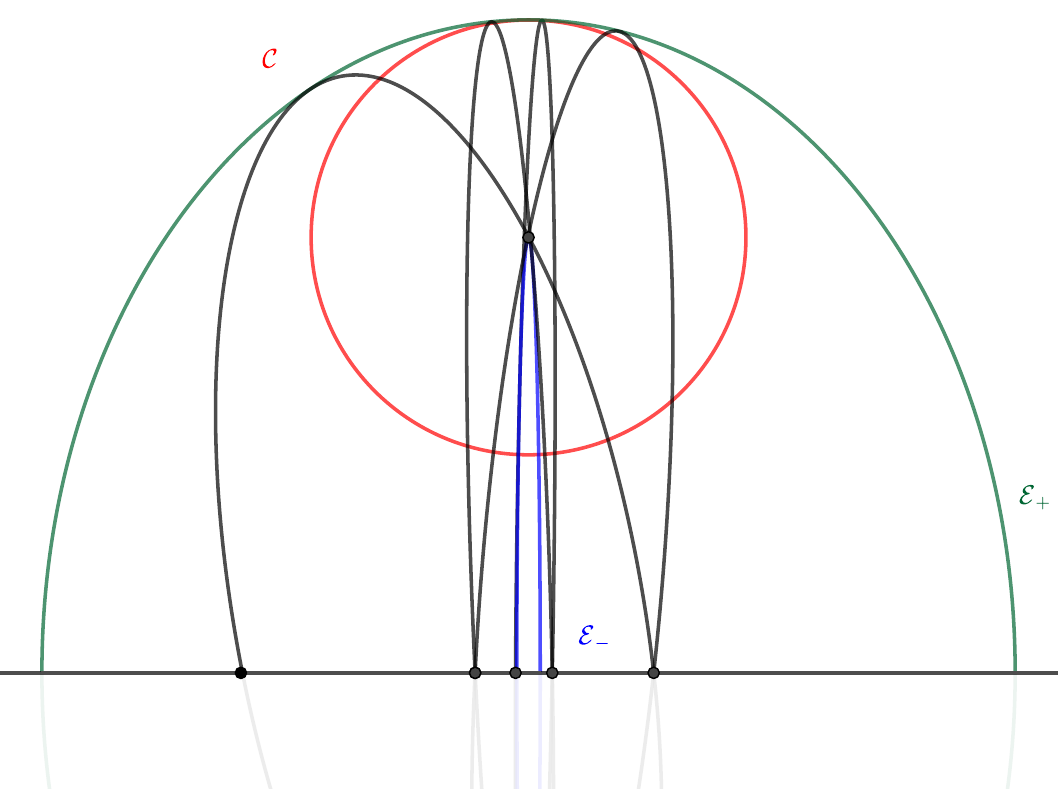}
	\end{tabular}
	\caption{Four iterations of the Boltzmann map with $(E,D)$ near $D=2$ with $-\frac{1}{2}\leq E < 0$ and $D<2$ (left) and $D>2$ (right). }
	\label{fig:Fomenko2}
\end{figure}
Near the upper boundary $1+2DE + 4E^2=0$, the radius of the circle $\C$ shrinks to 0, and the trajectories are trapped between the inner elliptic caustic $\E_-$ and the outer elliptic caustic $\E_+$. These trajectories limit to the 2-periodic trajectory which aligns with the upper half of the outer caustic $\E_+$, as shown in Figure \ref{fig:Fomenko3b}. 
\begin{figure}[ht]
	\centering
	\includegraphics[width=0.5\textwidth]{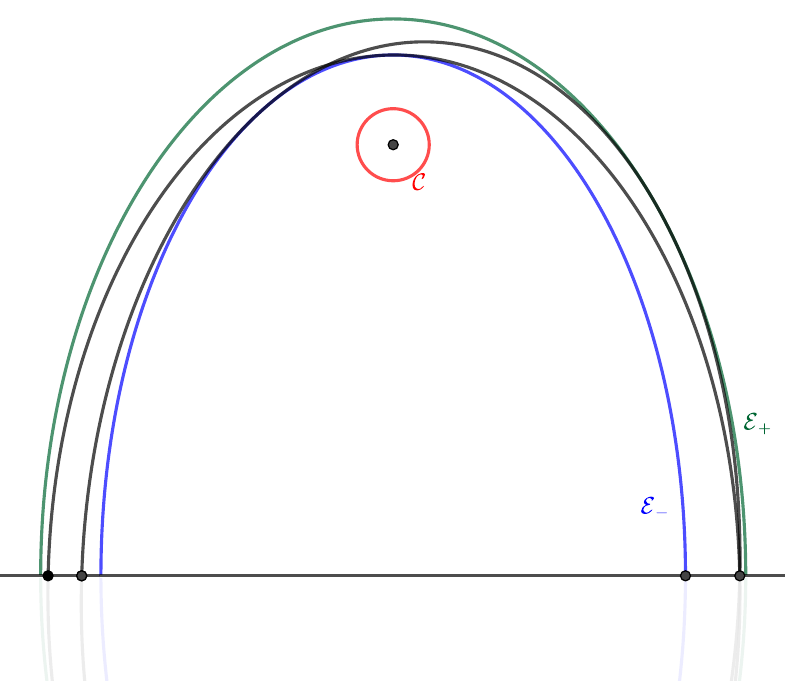} 
	\caption{Two iterations of the Boltzmann map with $(E,D)$ near $1+2DE+4E^2=0$ with $-\frac{1}{2}<E<0$. }
	\label{fig:Fomenko3b}
\end{figure}

This discussion shows that in this case, the Boltzmann system will be Liouville equivalent to the billiard within half-ellipse, as the Fomenko graph shown in the right-hand side of Figure \ref{FomenkoGraphs} is identical, see for example \cites{DR2009,Fokicheva2014}.
\end{proof}

\subsection*{Acknowledgment}
The research of M.~R.~ was supported
by the Australian Research Council, Discovery Project No.~DP190101838 \emph{Billiards within confocal quadrics and beyond}, and by the Science Fund of Serbia grant \emph{Integrability and Extremal Problems in Mechanics, Geometry and Combinatorics}, MEGIC, Grant No.~7744592.

\bibliographystyle{amsalpha}
\nocite{*}
\bibliography{References}

\end{document}